\documentclass[12pt]{article}
\usepackage[psamsfonts]{amssymb}
\usepackage{amsthm,amsmath}

\title{SYSTEM RELIABILITY AND WEIGHTED LATTICE POLYNOMIALS}

\author{Alexander Dukhovny\\
\small{Mathematics Department, San Francisco State University} \\
\small{San Francisco, CA 94132, USA} \\
\small{dukhovny[at]math.sfsu.edu} %
\and
Jean-Luc Marichal\\
\small{Institute of Mathematics, University of Luxembourg} \\
\small{162A, avenue de la Fa\"{\i}encerie, L-1511 Luxembourg, Luxembourg} \\
\small{jean-luc.marichal[at]uni.lu} %
}

\date{\small Revised, March 26, 2008}

\begin{document}
\maketitle

\theoremstyle{plain}
\newtheorem{theorem}{Theorem}
\newtheorem{lemma}[theorem]{Lemma}
\newtheorem{proposition}[theorem]{Proposition}
\newtheorem{corollary}[theorem]{Corollary}

\theoremstyle{definition}
\newtheorem{definition}[theorem]{Definition}
\newtheorem{example}[theorem]{Example}

\theoremstyle{remark}
\newtheorem*{conjecture}{\indent Conjecture}
\newtheorem*{remark}{\indent Remark}

\newcommand{\N}{\mathbb{N}}                     
\newcommand{\R}{\mathbb{R}}                     
\newcommand{\Vspace}{\vspace{2ex}}                  
\newcommand{\Ind}{\mathrm{Ind}}

\begin{abstract}
The lifetime of a system of connected units under some natural assumptions can be represented as a random variable $Y$ defined as a weighted
lattice polynomial of random lifetimes of its components. As such, the concept of a random variable $Y$ defined by a weighted lattice polynomial
of (lattice-valued) random variables is considered in general and in some special cases. The central object of interest is the cumulative
distribution function of $Y$. In particular, numerous results are obtained for lattice polynomials and weighted lattice polynomials in case of
independent arguments and in general. For the general case, the technique consists in considering the joint probability generating function of
``indicator'' variables. A connection is studied between $Y$ and order statistics of the set of arguments.
\end{abstract}


\section{Introduction}

In a system consisting of several interconnected units its lifetime can often be expressed in terms of those of the components. Assuming that
the system and the units are of the crisply ``on/off'' kind, the lifetime of a ``serially connected'' segment of units (where all units are
needed for functioning) is the minimum of the individual lifetimes; for a circuit of parallel units (one unit is sufficient for work) the
lifetime is the maximum of the individual lifetimes. In a complex system where both connection types are involved the system's lifetime $Y$ can
be expressed as a lattice polynomial function (combination of minima and maxima) of the component's lifetimes.

In addition, there may be ``collective upper bounds'' on lifetimes of certain subsets of units imposed by external conditions (say, physical
properties of the assembly) or even ``collective lower bounds'' imposed for instance by back-up blocks with constant lifetimes. These upper and
lower bounds can be easily modelled by incorporating additional components with constant lifetimes. Indeed, consider a subset with random
lifetime $X$ and assume that a lower (resp.\ upper) bound $c\in\R$ is required on this lifetime, which then becomes $c\vee X$ (resp.\ $c\wedge
X$). This situation can be modelled by connecting in parallel (resp. series) each component of that subset with a component of constant lifetime
$c$.

Following the above description, the lifetime of such a system involves the lower and upper bounds as follows. Denoting by $[n]=\{1,\ldots,n\}$
the set of indices of the units and by $X_i$ the random lifetime of the $i$th unit ($i\in [n]$), we conclude that
$$
Y = p(X_1,\ldots,X_n)
$$
where $p$ is an $n$-ary weighted lattice polynomial,\footnote{In the rest of the paper, \emph{weighted lattice polynomial functions}\/ will be
called \emph{weighted lattice polynomials}.} whose definition is recalled next (see Marichal \cite{Marc}). Let $L$ denote an arbitrary bounded
distributive lattice with lattice operations $\wedge$ and $\vee$ and denote by $a$ and $b$ its bottom and top elements. For any $\alpha,\beta\in
L$ and any subset $S\subseteq [n]$, let $\vec e^{\, \alpha,\beta}_S$ denote the characteristic vector of $S$ in $\{\alpha,\beta\}^n$, that is,
the $n$-tuple whose $i$th component is $\beta$, if $i\in S$, and $\alpha$, otherwise. The special case $\vec e^{\, 0,1}_S$ will be denoted $\vec
\varepsilon_S$.

\begin{definition}
The class of \emph{weighted lattice polynomials}\/ from $L^n$ to $L$ is defined as follows:
\begin{enumerate}
\item[(1)] For any $k\in [n]$ and any $c\in L$, the projection $(x_1,\ldots,x_n)\mapsto x_k$ and the constant function $(x_1,\ldots,x_n)\mapsto
c$ are weighted lattice polynomials from $L^n$ to $L$.

\item[(2)] If $p$ and $q$ are weighted lattice polynomials from $L^n$ to $L$, then $p\wedge q$ and $p\vee q$ are weighted lattice polynomials
from $L^n$ to $L$.

\item[(3)] Every weighted lattice polynomial from $L^n$ to $L$ is constructed by finitely many applications of the rules (1) and (2).
\end{enumerate}
\end{definition}

It was proved \cite{Marc} that any weighted lattice polynomial $p:L^n\to L$ can be expressed in disjunctive normal form : there exists a set
function $w:2^{[n]}\to L$ such that
\begin{equation}\label{eq:WlpDisjF}
p(\vec x)=\bigvee_{S\subseteq [n]}\Big[w(S)\wedge\bigwedge_{i\in S}x_i\Big].
\end{equation}
Moreover, from among all the possible set functions $w:2^{[n]}\to L$ fulfilling (\ref{eq:WlpDisjF}), only one is nondecreasing. It is defined as
\begin{equation}\label{eq:WlpDisjFw}
w(S)=p(\vec e^{\, a,b}_S),\qquad S\subseteq [n].
\end{equation}

The $n$-ary weighted lattice polynomial defined by a given nondecreasing set function $w:2^{[n]}\to L$ will henceforth be denoted $p_w$.

\begin{remark}
Weighted lattice polynomials are strongly related to the concept of discrete \emph{Sugeno integral}\/ \cite{Sug74,Sug77}. Indeed, as observed in
\cite{Marc}, a function $f:L^n\to L$ is a discrete Sugeno integral if and only if it is a weighted lattice polynomial fulfilling $f(\vec e^{\,
a,b}_{\varnothing})=a$ and $f(\vec e^{\, a,b}_{[n]})=b$.
\end{remark}

Assuming that $L\subseteq\overline{\R}=[-\infty,\infty]$ and that the lifetime $X_i$ of every unit $i\in [n]$ is $L$-valued, the lifetime of the
system, with possible collective lower or upper bounds, is then given by
$$
Y = p_w(X_1,\ldots,X_n)
$$
where $p_w$ is a weighted lattice polynomial from $L^n$ to $L$.

In those terms, $\Pr[Y>y]$ is the probability that by the time $y$ the system is still on, that is, the \emph{system reliability}\/ at time $y$.
This probability is in turn determined by the pattern of units that are on at the time $y$, that is, by the indicators of events $X_i>y$.

In this paper we investigate the cumulative distribution function (c.d.f.) of $Y$ in case of general dependent arguments (see Section 2). Our
results generalize both the case of independent arguments, which was studied in Marichal \cite{Mar08}, and the special situation where no
constant lifetimes are considered (the weighted lattice polynomial is merely a lattice polynomial), which was investigated by Marichal
\cite{Mar06} for independent arguments and then by Dukhovny \cite{Duk07} for general dependent arguments.

We also investigate the special case where the weighted lattice polynomial is a symmetric function and the case where, at any time, the
reliability of any group depends only on the number of units in the group (see Section 3).

\section{Cumulative distribution function in the general case}

We assume throughout that $L\subseteq\overline{\R}$, which implies that $L$ is totally ordered. For an $L$-valued random variable $X$ one can
introduce a \emph{supplementary variable}
$$
\chi(x)=\Ind(X>x)\in\{0,1\}.
$$

For a set of $L$-valued random variables $X_i$, $i\in [n]$, we consider a vector of (synchronous) indicator variables
$$
\vec\chi(x)=(\chi_1(x),\ldots,\chi_n(x)),
$$
where
$$
\chi_i(x)=\Ind(X_i>x),\qquad i\in [n].
$$

For any $\vec\kappa\in\{0,1\}^n$, let $S_{\vec\kappa}$ denote the subset of $[n]$ whose characteristic vector in $\{0,1\}^n$ is $\vec\kappa$
(that is such that $\vec\varepsilon_{S_{\vec\kappa}}=\vec\kappa$).

We have the following fundamental lemma:

\begin{lemma}\label{lemma:WLPindVar}
Let $p_w:L^n\to L$ be a weighted lattice polynomial and $Y=p_w(X_1,\ldots,X_n)$. Then
$$
\Ind(Y>y)=\Ind(w(S_{\vec\chi(y)})>y),\qquad y\in L.
$$
\end{lemma}

\begin{proof}
The proof mainly lies on the distributive property of $\Ind(\cdot)$ with respect to disjunction and conjunction: for any events $A$ and $B$,
\begin{eqnarray*}
\Ind(A\vee B) &=& \Ind(A)\vee\Ind(B),\\
\Ind(A\wedge B) &=& \Ind(A)\wedge\Ind(B).
\end{eqnarray*}
By (\ref{eq:WlpDisjF}), we obtain
\begin{eqnarray*}
\Ind(Y>y) &=& \Ind(p_w(X_1,\ldots,X_n)>y) \\
&=& \bigvee_{S\subseteq [n]}\Big(\Ind(w(S)>y)\wedge\bigwedge_{i\in S}\Ind(X_i>y)\Big) \\
&=& \pi_{\omega_y}(\vec\chi(y)),
\end{eqnarray*}
where $\pi_{\omega_y}:\{0,1\}^n \to\{0,1\}$ is the weighted lattice polynomial defined by the nondecreasing set function
$\omega_y:2^{[n]}\to\{0,1\}$ fulfilling
$$
\omega_y(S)=\Ind(w(S)>y),\qquad S\subseteq [n].
$$

Finally, since $\vec\chi(y)\in\{0,1\}^n$, by (\ref{eq:WlpDisjFw}) we have
$$
\pi_{\omega_y}(\vec\chi(y)) = \omega_y(S_{\vec\chi(y)}) = \Ind(w(S_{\vec\chi(y)})>y).
$$
\end{proof}

\begin{remark}
Interestingly enough, we have observed in the proof of Lemma~\ref{lemma:WLPindVar} that the identity $Y=p_w(X_1,\ldots,X_n)$ has its binary
counterpart in terms of indicator variables, namely
$$
\chi_Y(y)=\pi_{\omega_y}(\chi_1(y),\ldots,\chi_n(y)),\qquad y\in L,
$$
where $\chi_Y(y)=\Ind(Y>y)$.
\end{remark}

We now state our main result.

\begin{theorem}\label{thm:OmFes}
Let $Y=p_w(X_1,\ldots,X_n)$ and denote its c.d.f.\ by $F_Y(y)$. Then, for any $y\in L$,
\begin{equation}\label{eq:OmFeS}
1-F_Y(y)=\sum_{\vec\kappa\in\{0,1\}^n} \Ind(w(S_{\vec\kappa})>y)\,\Pr(\vec\chi(y)=\vec\kappa),
\end{equation}
or equivalently,
\begin{equation}\label{eq:OmFeS2}
1-F_Y(y)=\sum_{S\subseteq [n]} \Ind(w(S)>y)\,\Pr(\vec\chi(y)=\vec\varepsilon_S).
\end{equation}
\end{theorem}

\begin{proof}
By Lemma~\ref{lemma:WLPindVar}, we have
\begin{eqnarray*}
1-F_Y(y) &=& \mathrm{E}(\Ind(Y>y)) = \mathrm{E}(\Ind(w(S_{\vec\chi(y)})>y))\\
&=& \sum_{\vec\kappa\in\{0,1\}^n} \mathrm{E}(\Ind(w(S_{\vec\kappa})>y))\,\Pr(\vec\chi(y)=\vec\kappa)\\
&=& \sum_{\vec\kappa\in\{0,1\}^n} \Ind(w(S_{\vec\kappa})>y)\,\Pr(\vec\chi(y)=\vec\kappa).
\end{eqnarray*}
\end{proof}

Also, since
$$
\Ind(w(S_{\vec\kappa})>y)=1-\Ind(w(S_{\vec\kappa})\leqslant y),
$$
and $\sum_{\vec\kappa} \Pr(\vec\chi(y)=\vec\kappa)=1$, one can present (\ref{eq:OmFeS}) as
$$
F_Y(y)=\sum_{\vec\kappa\in\{0,1\}^n} \Ind(w(S_{\vec\kappa})\leqslant y)\,\Pr(\vec\chi(y)=\vec\kappa).
$$

The formulas of Theorem~\ref{thm:OmFes} are quite remarkable. They show that the vector of binary status indicators $\vec\chi(y)$ at the time
$y$ is the only thing needed to determine the status of the whole system. Moreover, the ``dot product'' nature of these formulas has, as it will
be seen, a special significance.

\begin{remark}
It is noteworthy that an alternative form of formula (\ref{eq:OmFeS2}), namely
$$
1-F_Y(y)=\sum_{S\subseteq [n]} \Ind(w(S)>y)\,\Pr\Big(\bigvee_{i\in [n]\setminus S} X_i\leqslant y<\bigwedge_{i\in S}X_i\Big),
$$
was already found through a direct use of the disjunctive normal form (\ref{eq:WlpDisjF}) of $p_w$ (see proof of Theorem 1 in Marichal
\cite{Mar08}).
\end{remark}

In the general case when the random arguments are dependent the probability parts of the formulas (\ref{eq:OmFeS})--(\ref{eq:OmFeS2}) can be
treated using the technique of probability generating functions (p.g.f.). For a random variable $X$ with a c.d.f.\ $F(x)$ we introduce the
p.g.f.\ of its indicator $\chi(x)$:
\begin{eqnarray*}
G(z,x) &=& \mathrm{E}(z^{\chi(x)}) ~=~ \Pr[\chi(x)=0]+z\Pr[\chi(x)=1]\\
&=& F(x)+z[1-F(x)],\qquad |z|\leqslant 1.
\end{eqnarray*}

For a set of (generally dependent) random variables $X_1,\ldots,X_n$, we denote the joint p.g.f.\ of their indicator variables $\vec\chi(x)$ by
$$
G(\vec z,x)=\mathrm{E}(\vec z^{\,\vec\chi(x)})=\mathrm{E}\Big(\prod_{i=1}^n z_i^{\chi_i(x)}\Big),\qquad |z_i|\leqslant 1,\, i\in [n].
$$

By definition,
\begin{equation}\label{eq:GMultPol}
G(\vec z,x)=\sum_{\vec\kappa\in\{0,1\}^n}\Pr(\vec\chi(x)=\vec\kappa)\,\vec z^{\,\vec\kappa}=\sum_{S\subseteq
[n]}\Pr(\vec\chi(x)=\vec\varepsilon_S)\,\prod_{i\in S}z_i.
\end{equation}
Since $G(\vec z,x)$ is a multilinear polynomial in $z_1,\ldots,z_n$, we immediately obtain~\footnote{Note that, since $G(\vec z,x)$ is a
multilinear polynomial, its derivatives can be also obtained by using discrete derivatives (see Grabisch et al.~\cite[\S 5]{GraMarRou00}).}
$$
\Pr(\vec\chi(x)=\vec\kappa)=\frac{\partial^{(\vec\kappa)}}{\partial\vec z}\, G(\vec z,x)\Big|_{\vec z=\vec 0}
$$
which, with (\ref{eq:OmFeS}), enables to compute $F_Y(y)$ whenever $G(\vec z,y)$ can be provided.

As a corollary, we immediately retrieve the known expression of $F_Y(y)$ in the special case when $X_1,\ldots,X_n$ are independent (see Marichal
\cite{Mar08}).

\begin{corollary}\label{cor:IndVar}
When $X_1,\ldots,X_n$ are independent random variables with c.d.f.'s $F_i(x)$, $i\in [n]$, we have
$$
1-F_Y(y)=\sum_{\vec\kappa\in\{0,1\}^n} \Ind(w(S_{\vec\kappa})>y)\,\prod_{i=1}^n F_i(y)^{1-\kappa_i}[1-F_i(y)]^{\kappa_i}.
$$
\end{corollary}

\begin{proof}
By independence, we immediately have
$$
G(\vec z,y) = \prod_{i=1}^n [F_i(y)+z_i(1-F_i(y))],
$$
and hence
$$
\frac{\partial^{(\vec\kappa)}}{\partial\vec z}\, G(\vec z,y)\Big|_{\vec z=\vec 0} = \prod_{i=1}^n F_i(y)^{1-\kappa_i}[1-F_i(y)]^{\kappa_i}.
$$
\end{proof}

\begin{remark}
The assumptions of Corollary~\ref{cor:IndVar} can be slightly relaxed while keeping a product form for $\Pr(\vec\chi(x)=\vec\kappa)$. In fact,
assuming that the arguments are dependent but their synchronous indicators $\chi_1(x),\ldots,\chi_n(x)$ are independent for any fixed $x$, we
get
$$
\Pr(\vec\chi(x)=\vec\kappa)=\prod_{i=1}^n \Pr[\chi_i(x)=0]^{1-\kappa_i}\Pr[\chi_i(x)=1]^{\kappa_i}.
$$
Similarly, in the case when all $X_i=Z+Y_i$, $Z$ and all $Y_i$ being independent (random shift), the probability density function formula is the
one for the independent arguments case convoluted with the density function of $Z$.
\end{remark}

An alternative approach for computing $\Pr(\vec\chi(x)=\vec\varepsilon_S)$ is given in the following proposition, where $F(\vec x)$ denotes the
joint c.d.f.\ of $X_1,\ldots,X_n$:

\begin{proposition}\label{prop:GMobTrans}
For any fixed $x\in L$, the set function $S\mapsto\Pr(\vec\chi(x)=\vec\varepsilon_S)$ is the M\"obius transform of the set function $S\mapsto
G(\vec\varepsilon_S,x)$ and we have
$$
G(\vec\varepsilon_S,x)=F(\vec e^{\, x,b}_S), \qquad S\subseteq [n].
$$
\end{proposition}

\begin{proof}
Let $x\in L$. By (\ref{eq:GMultPol}), we have
$$
G(\vec\varepsilon_S,x) = \sum_{T\subseteq S}\Pr(\vec\chi(x)=\vec\varepsilon_T),
$$
which establishes the first part of the result. From this latter identity, we obtain
\begin{eqnarray*}
G(\vec\varepsilon_S,x) &=& \sum_{T\subseteq S}\Pr(X_i\leqslant x~\forall i\in [n]\setminus T ~\mbox{and} ~ X_i>x~\forall i\in T)\\
&=& \Pr(X_i\leqslant x~\forall i\in [n]\setminus S),
\end{eqnarray*}
which proves the second part.
\end{proof}

The result in Proposition~\ref{prop:GMobTrans} yields an explicit form of $G(\vec z,x)$ in terms of $F(\vec e^{\, x,b}_S)$, namely
$$
G(\vec z,x) = \sum_{S\subseteq [n]} F(\vec e^{\, x,b}_S)\, \prod_{i\in S}z_i\,\prod_{i\in [n]\setminus S}(1-z_i).
$$
Moreover, combining Proposition~\ref{prop:GMobTrans} and M\"obius inversion formula, it follows immediately that
\begin{eqnarray}
\Pr(\vec\chi(x)=\vec\varepsilon_S) &=& \sum_{T\subseteq S} (-1)^{|S|-|T|}\, G(\vec\varepsilon_T,x)\nonumber\\
&=& \sum_{T\subseteq S} (-1)^{|S|-|T|}\, F(\vec e^{\, x,b}_T). \label{eq:PrFeT}
\end{eqnarray}

\begin{remark}
We have observed that, since $G(\vec z,y)$ is a multilinear polynomial, it is completely determined by the $2^n$ values
$G(\vec\varepsilon_S,y)=F(\vec e^{\, y,b}_S)$, $S\subseteq [n]$. Thus, the computation of $F_Y(y)$ does not require the complete knowledge of
the joint c.d.f.\ $F(\vec x)$, but only of the values $F(\vec e^{\, y,b}_S)$ for all $y\in L$ and all $S\subseteq [n]$.
\end{remark}

\section{Selected special cases}

In this final section we consider three special cases that have a natural interpretation in terms of eligible applications:
\begin{enumerate}
\item The weighted lattice polynomial is a lattice polynomial (there are no lower or upper bounds).

\item The weighted lattice polynomial is a symmetric function (the system lifetime is a symmetric function of the component lifetimes).

\item The synchronous indicator variables are exchangeable (at any time, the reliability of any group depends only on the number of units in the
group).
\end{enumerate}

For any binary vector $\vec\kappa\in\{0,1\}^n$, we set $|\vec\kappa|=\sum_{i=1}^n\kappa_i$.

\subsection{Lattice polynomials}

An important special case arises when the lifetime of a system is obtained as a lattice polynomial of units' lifetimes. Recall that lattice
polynomials are defined as follows (see \cite{Marc}):

\begin{definition}
The class of \emph{lattice polynomials}\/ from $L^n$ to $L$ is defined as follows:
\begin{enumerate}
\item[(1)] For any $k\in [n]$, the projection $(x_1,\ldots,x_n)\mapsto x_k$ is a lattice polynomial from $L^n$ to $L$.

\item[(2)] If $p$ and $q$ are lattice polynomials from $L^n$ to $L$, then $p\wedge q$ and $p\vee q$ are lattice polynomials from $L^n$ to $L$.

\item[(3)] Every lattice polynomial from $L^n$ to $L$ is constructed by finitely many applications of the rules (1) and (2).
\end{enumerate}
\end{definition}

Thus, a lattice polynomial from $L^n$ to $L$ is nothing else than a weighted lattice polynomial $p_w:L^n\to L$ with a set function
$w:2^{[n]}\to\{a,b\}$ fulfilling $w(\varnothing)=a$ and $w([n])=b$ (see \cite{Marc}). In this case, for any $y\in L$, we obtain
\begin{eqnarray}
\Ind(w(S)>y) &=& \Ind(p_w(\vec e^{\, a,b}_S)>y) ~=~ \Ind(p_w(\vec e^{\, a,b}_S)=b)\,\Ind(b>y)\nonumber\\
&=& \pi_{\omega}(\vec\varepsilon_S) \, \Ind(b>y),\label{eq:IndPi}
\end{eqnarray}
where $\pi_{\omega}:\{0,1\}^n \to\{0,1\}$ is the (weighted) lattice polynomial defined by the nondecreasing set function
$\omega:2^{[n]}\to\{0,1\}$ fulfilling
$$
\omega(S)=\Ind(p_w(\vec e^{\, a,b}_S)=b),\qquad S\subseteq [n].
$$

\begin{remark}
Clearly, $\pi_{\omega}$ is exactly the weighted lattice polynomial $\pi_{\omega_y}$ introduced in the proof of Lemma~\ref{lemma:WLPindVar}
computed for the case of lattice polynomials. Moreover, it has the same disjunctive normal form as $p_w$.
\end{remark}

Since $\Pr(\vec\chi(b)=\vec\kappa)=0$ except when $\vec\kappa=\vec 0$, we obtain the following corollary (see Dukhovny \cite{Duk07}):

\begin{corollary}\label{cor:LP}
Under the assumptions of Theorem~\ref{thm:OmFes} and assuming further that $p_w$ is a lattice polynomial, we have
\begin{equation}\label{eq:corLP}
1-F_Y(y)=\sum_{\vec\kappa\in\{0,1\}^n} \pi_{\omega}(\vec\kappa)\,\Pr(\vec\chi(y)=\vec\kappa)
\end{equation}
in case of generally dependent argument and, in case of independence,
$$
1-F_Y(y)=\sum_{\vec\kappa\in\{0,1\}^n} \pi_{\omega}(\vec\kappa)\,\prod_{i=1}^n F_i(y)^{1-\kappa_i}[1-F_i(y)]^{\kappa_i}.
$$
\end{corollary}

\subsection{Symmetric weighted lattice polynomials}

Considering a system with collective upper or lower bounds on subset lifetimes, an important special case arises where the corresponding
weighted lattice polynomial is a symmetric function, that is, invariant under permutation of its variables.

Prominent instances of symmetric weighted lattice polynomials are given by the so-called \emph{order statistic} functions, whose definition is
recalled next; see for instance Ovchinnikov~\cite{Ovc96}.

\begin{definition}\label{de:OS}
For any $k\in [n]$, the $k$th \emph{order statistic}\/ function is the lattice polynomial $f_k:L^n\to L$ defined as
$$
f_k(\vec x)=x_{(k)}=\bigvee_{\textstyle{S\subseteq [n]\atop |S|= n-k+1}}\bigwedge_{i\in S}x_i.
$$
(For future needs we also formally define $f_0\equiv a$ and $f_{n+1}\equiv b$.)
\end{definition}

\begin{remark}
A system whose underlying weighted lattice polynomial is the order statistic function $f_{n-k+1}$ is called a \emph{$k$-out-of-$n$ system}; see
for instance the book by Rausand and H{\o}yland \cite[\S3.10]{RauHoy04}.
\end{remark}

The following proposition shows that a weighted lattice polynomial $p_w:L^n\to L$ is symmetric if and only if its underlying set function
$w:2^{[n]}\to\{a,b\}$, as defined in (\ref{eq:WlpDisjFw}), is \emph{cardinality-based}, that is, such that
$$
|S|=|S'| \quad\Rightarrow\quad w(S)=w(S').
$$
Equivalently, there exists a nondecreasing function $m:\{0,1,\ldots,n\}\to L$ such that
$$
w(S)=m(|S|), \qquad S\subseteq [n].
$$

The proposition also provides an expression of an arbitrary symmetric weighted lattice polynomial as a ``convolution'' of $L$-valued functions
on $[n]$: $m(s)$ and $l(s)=f_s(\vec x)$.

\begin{proposition}\label{prop:SymWLP}
Consider a weighted lattice polynomial $p_w:L^n\to L$. Then the following assertions are equivalent:
\begin{enumerate}
\item[(i)] The set function $w$ is cardinality-based.

\item[(ii)] There exists a nondecreasing function $m:\{0,1,\ldots,n\}\to L$ such that
$$
p_w(\vec x) = \bigvee_{s=0}^n \Big[m(s)\wedge f_{n-s+1}(\vec x)\Big].
$$

\item[(iii)] $p_w$ is symmetric.
\end{enumerate}
\end{proposition}

\begin{proof}
$(i)\Rightarrow (ii)$ By (\ref{eq:WlpDisjF}), we have
$$
p_w(\vec x) = \bigvee_{S\subseteq [n]}\Big[w(S)\wedge\bigwedge_{i\in S}x_i\Big] = \bigvee_{s=0}^n m(s)\wedge\bigg[\bigvee_{\textstyle{S\subseteq
[n]\atop |S|=s}}\bigwedge_{i\in S}x_i\bigg],
$$
which proves the stated formula.

$(ii)\Rightarrow (iii)$ Trivial.

$(iii)\Rightarrow (i)$ Follows immediately from (\ref{eq:WlpDisjFw}).
\end{proof}

\begin{remark}
It is noteworthy that Definition~\ref{de:OS} and Proposition~\ref{prop:SymWLP} can also be stated in the more general case where $L$ is an
arbitrary bounded distributive lattice.
\end{remark}

For symmetric weighted lattice polynomials, Theorem~\ref{thm:OmFes} takes the following form. For any $y\in L$, set $s(y)=\min\{s,n+1\, :\,
m(s)>y\}$.

\begin{theorem}
Under the assumptions of Theorem~\ref{thm:OmFes} and assuming further that the set function $w$ is cardinality-based, we have
\begin{equation}\label{eq:CardBaFY}
1-F_Y(y) = \sum_{s=s(y)}^n \Pr(|\vec\chi(y)|=s).
\end{equation}
\end{theorem}

\begin{proof}
When $w$ is cardinality-based, the sum in (\ref{eq:OmFeS}) can be rewritten as
\begin{eqnarray}
1-F_Y(y) &=& \sum_{s=0}^n \Ind(m(s)>y)\sum_{\textstyle{\vec\kappa\in\{0,1\}^n \atop |\vec\kappa |=s}}\Pr(\vec\chi(y)=\vec\kappa)\nonumber\\
&=& \sum_{s=0}^n \Ind(m(s)>y)\Pr(|\vec\chi(y)|=s).\label{eq:CDFcb}
\end{eqnarray}
When $m(n)\leqslant y$, we have $s(y)=n+1$, and (\ref{eq:CDFcb}) reduces to 0.
\end{proof}

Obviously, (\ref{eq:CardBaFY}) can also be presented as
\begin{equation}\label{eq:CardBaFY2}
F_Y(y) = \sum_{s=0}^{s(y)-1}\Pr(|\vec\chi(y)|=s).
\end{equation}

Both (\ref{eq:CardBaFY}) and (\ref{eq:CardBaFY2}) require the knowledge of $\Pr(|\vec\chi(y)|=s)$, which can be calculated through the following
formula:

\begin{proposition}
For any $s\in\{0,1,\ldots,n\}$, we have
$$
\Pr(|\vec\chi(y)|=s) = \frac{1}{s!}\,\frac{\mathrm{d}^s}{\mathrm{d}t^s}\, G(t\vec 1,y)\Big|_{t=0}.
$$
\end{proposition}

\begin{proof}
From (\ref{eq:GMultPol}) it follows that
\begin{eqnarray*}
G(t\vec 1,y) &=& \sum_{\vec\kappa\in\{0,1\}^n}\Pr(\vec\chi(y)=\vec\kappa)\, t^{|\vec\kappa |} ~ = ~
\sum_{s=0}^n\bigg[\sum_{\textstyle{\vec\kappa\in\{0,1\}^n \atop |\vec\kappa | = s}}\Pr(\vec\chi(y)=\vec\kappa)\bigg]t^s\\
&=& \sum_{s=0}^n \Pr(|\vec\chi(y)|=s)\, t^s,
\end{eqnarray*}
from which the result follows.
\end{proof}

In terms of the joint c.d.f.\ of $X_1,\ldots,X_n$, we also have the following formula:

\begin{proposition}
For any $s\in\{0,1,\ldots,n\}$, we have
\begin{equation}\label{eq:PrsFe}
\Pr(|\vec\chi(y)|=s) = \sum_{\textstyle{T\subseteq [n]\atop |T|\leqslant s}} (-1)^{s-|T|}{n-|T|\choose s-|T|}\, F(\vec e^{\, y,b}_T).
\end{equation}
\end{proposition}

\begin{proof}
From (\ref{eq:PrFeT}) it follows that
\begin{eqnarray*}
\Pr(|\vec\chi(y)|=s) &=& \sum_{\textstyle{S\subseteq [n]\atop |S|=s}} \Pr(\vec\chi(y)=\vec\varepsilon_S)
~=~ \sum_{\textstyle{S\subseteq [n]\atop |S|=s}} \sum_{T\subseteq S} (-1)^{|S|-|T|}\, F(\vec e^{\, y,b}_T)\\
&=& \sum_{\textstyle{T\subseteq [n]\atop |T|\leqslant s}} F(\vec e^{\, y,b}_T)\, \sum_{\textstyle{S\supseteq T\atop |S|=s}} (-1)^{|S|-|T|},
\end{eqnarray*}
where the inner sum reduces to $(-1)^{s-|T|}{n-|T|\choose s-|T|}$.
\end{proof}

We now derive an expression of $\Pr(|\vec\chi(y)|=s)$ is terms of the c.d.f.'s of the order statistic functions.

Consider the random variable $X_{(k)}=f_k(X_1,\ldots,X_n)$ and its c.d.f.\ $F_{(k)}(y)=F_{X_{(k)}}(y)$. It is easy to see that, for the $k$th
order statistic function $f_k$, we have $s(y)=n-k+1$ if $y<b$ and $s(y)=n+1$ if $y=b$. Therefore, from (\ref{eq:CardBaFY2}) it follows that
\begin{equation}\label{eq:FkPr}
F_{(k)}(y) = \sum_{s=0}^{n-k}\Pr(|\vec\chi(y)|=s) = \Pr(|\vec\chi(y)|\leqslant n-k),
\end{equation}
from which we immediately derive the following formula:~\footnote{In (\ref{eq:FkPr}) we formally define $F_{(0)}(y)\equiv 1$ and
$F_{(n+1)}(y)\equiv 0$.}

\begin{proposition}\label{prop:PrDiff}
For any $s\in\{0,1,\ldots,n\}$, we have
\begin{equation}\label{eq:PrDiff}
\Pr(|\vec\chi(y)|=s)=F_{(n-s)}(y)-F_{(n-s+1)}(y).
\end{equation}
\end{proposition}

Incidentally, by combining (\ref{eq:PrsFe}) and (\ref{eq:FkPr}), we get an explicit form of $F_{(k)}(y)$ in terms of $F(\vec x)$; see also David
and Nagaraja~\cite[\S 5.3]{DavNag03}.

\begin{corollary}
We have
$$
F_{(k)}(y) = \sum_{\textstyle{S\subseteq [n]\atop |S|\geqslant k}} (-1)^{|S|-k}{|S|-1\choose k-1}\, F(\vec e^{\, y,b}_{[n]\setminus S}).
$$
\end{corollary}

\begin{proof}
Combining (\ref{eq:PrsFe}) and (\ref{eq:FkPr}), we obtain
\begin{eqnarray*}
F_{(k)}(y) &=& \sum_{s=0}^{n-k}\sum_{t=0}^s\sum_{\textstyle{T\subseteq [n]\atop |T|=t}} (-1)^{s-t}{n-t\choose s-t}\, F(\vec e^{\, y,b}_T)\\
&=& \sum_{t=0}^{n-k}\sum_{\textstyle{T\subseteq [n]\atop |T|=t}} F(\vec e^{\, y,b}_T)\, \sum_{s=t}^{n-k}(-1)^{s-t}{n-t\choose s-t},
\end{eqnarray*}
where the inner sum reduces to the following classical binomial identity
$$
\sum_{s=0}^{n-k-t} (-1)^s {n-t\choose s} = (-1)^{n-k-t}{n-t-1\choose k-1}.
$$
Therefore,
$$
F_{(k)}(y) = \sum_{\textstyle{T\subseteq [n]\atop |T|\leqslant n-k}} (-1)^{n-k-|T|}{n-|T|-1\choose k-1}\, F(\vec e^{\, y,b}_T),
$$
and the result follows from the substitution $S=[n]\setminus T$.
\end{proof}

\subsection{Cardinality symmetry of the indicator variables}

A case is now considered where the joint c.d.f.\ of the arguments $X_1,\ldots,X_n$ is such that, for any $x\in L$, $\Pr(\vec\chi(x)=\vec\kappa)$
depends only on $|\vec\kappa|$, that is
$$
|\vec\kappa|=|\vec\kappa'| \quad\Rightarrow\quad \Pr(\vec\chi(x)=\vec\kappa)=\Pr(\vec\chi(x)=\vec\kappa').
$$
Equivalently, for any $x\in L$, there exists a function $g_x:\{0,1,\ldots,n\}\to [0,1]$ such that
\begin{equation}\label{eq:CardBaCDF}
\Pr(\vec\chi(x)=\vec\kappa)=g_x(|\vec\kappa|), \qquad \vec\kappa\in\{0,1\}^n.
\end{equation}
In terms of the analogy with system reliability, assumption (\ref{eq:CardBaCDF}) means that the probability that a group of units survives
beyond $x$ (that is, the reliability of this group at time $x$) depends only on the number of units in the group, which is why we call
assumption (\ref{eq:CardBaCDF}) ``cardinality symmetry''. In view of (\ref{eq:PrFeT}), this assumption is satisfied, for instance, when
$X_1,\ldots,X_n$ are exchangeable, that is, when their joint c.d.f.\ is invariant under any permutation of indices. However,
(\ref{eq:CardBaCDF}) itself suggests only that the synchronous indicator variables be exchangeable.

In particular, for independent and identically distributed (i.i.d.) arguments with a common c.d.f.\ $F(x)$, we have (see
Corollary~\ref{cor:IndVar}):
$$
\Pr(\vec\chi(x)=\vec\kappa) = F(x)^{n-|\vec\kappa|}[1-F(x)]^{|\vec\kappa|}.
$$
More generally, one can easily see that when the arguments $X_1,\ldots,X_n$ are obtained from some i.i.d.\ random variables $Y_1,\ldots,Y_n$ by
a transformation
$$
X_i=f(Y_i,\vec Z),\qquad i\in [n],
$$
where $\vec Z$ is a vector of some (random) parameters, assumption (\ref{eq:CardBaCDF}) still holds true.

\begin{corollary}
Under the assumptions of Theorem~\ref{thm:OmFes} and assuming further that the joint c.d.f.\ possesses property (\ref{eq:CardBaCDF}), we have
\begin{equation}\label{eq:CSCor}
1-F_Y(y)=\sum_{s=0}^n g_y(s) \sum_{\textstyle{S\subseteq [n]\atop |S|=s}} \Ind(w(S)>y).
\end{equation}
In particular, in case of lattice polynomials,
$$
1-F_Y(y)=\sum_{s=0}^n g_y(s) \sum_{\textstyle{\vec\kappa\in\{0,1\}^n\atop |\vec\kappa|=s}} \pi_{\omega}(\vec\kappa).
$$
\end{corollary}

By Proposition~\ref{prop:GMobTrans}, we also have
$$
F(\vec e^{\, x,b}_S)=\sum_{t=0}^{|S|}{|S|\choose t}\, g_x(t),
$$
which shows that the set function $S\mapsto F(\vec e^{\, x,b}_S)$ is cardinality-based.

For order statistic functions, we have~\footnote{The total probability condition allows to formally extend (\ref{eq:FkPr2}) to the case $k=0$.}
\begin{equation}\label{eq:FkPr2}
F_{(k)}(y) = \sum_{s=0}^{n-k}{n\choose s}\, g_y(s).
\end{equation}
Indeed, based on (\ref{eq:FkPr}), we have
$$
F_{(k)}(y) = \Pr(|\vec\chi(y)|\leqslant n-k) = \sum_{\textstyle{\vec\kappa\in\{0,1\}^n\atop |\vec\kappa|\leqslant n-k}} g_y(|\vec\kappa|),
$$
which proves (\ref{eq:FkPr2}).

It turns out that cardinality symmetry is necessary and sufficient for the special relation between the c.d.f.'s of $Y$ and order statistic
functions provided in Marichal \cite[Theorem 8]{Mar06} for the case of a lattice polynomial of i.i.d.\ arguments. Beyond its theoretical value,
the relation is a simple linear equation and can be used to facilitate computations.

\begin{theorem}
Under the assumptions of Theorem~\ref{thm:OmFes}, the following relations hold true for an arbitrary weighted lattice polynomial if and only if
the arguments $X_1,\ldots,X_n$ possess cardinality symmetry:
\begin{eqnarray}
1-F_Y(y) &=& \sum_{s=0}^n\overline{\omega}_s\,[F_{(n-s)}(y)-F_{(n-s+1)}(y)],\label{eq:LT1}\\
\overline{\omega}_s &=& \frac 1{{n\choose s}}\sum_{\textstyle{S\subseteq [n]\atop |S|=s}} \Ind(w(S)>y),\qquad s=0,\ldots,n.\label{eq:LT2}
\end{eqnarray}
\end{theorem}

\begin{proof}
$(\Leftarrow)$ Assuming (\ref{eq:CardBaCDF}), on the strength of (\ref{eq:FkPr2}),
\begin{equation}\label{eq:LTfE}
{n\choose s}\, g_y(s) = F_{(n-s)}(y)-F_{(n-s+1)}(y),\qquad s=0,\ldots,n.
\end{equation}
Now, using (\ref{eq:LTfE}) in (\ref{eq:CSCor}), we obtain (\ref{eq:LT1})--(\ref{eq:LT2}).

$(\Rightarrow)$ Assume that (\ref{eq:LT1}) and (\ref{eq:LT2}) are true for any weighted lattice polynomial and prove that (\ref{eq:CardBaCDF})
must hold true.

Consider the particular case of a lattice polynomial $p_w$. Substituting (\ref{eq:IndPi}) and (\ref{eq:PrDiff}) in
(\ref{eq:LT1})--(\ref{eq:LT2}) we obtain
\begin{equation}\label{eq:tzu1}
1-F_Y(y)=\sum_{s=0}^n\frac 1{{n\choose s}}\sum_{\textstyle{\vec\kappa\in\{0,1\}^n\atop
|\vec\kappa|=s}}\pi_\omega(\vec\kappa)\sum_{\textstyle{\vec\kappa\in\{0,1\}^n\atop |\vec\kappa|=s}}\Pr(\vec\chi(y)=\vec\kappa).
\end{equation}
On the other hand, by (\ref{eq:corLP}) we have
\begin{equation}\label{eq:tzu2}
1-F_Y(y)=\sum_{s=0}^n\sum_{\textstyle{\vec\kappa\in\{0,1\}^n\atop |\vec\kappa|=s}} \pi_\omega(\vec\kappa)\,\Pr(\vec\chi(y)=\vec\kappa)
\end{equation}
Therefore, the right-hand sides of (\ref{eq:tzu1}) and (\ref{eq:tzu2}) coincide for any lattice polynomial $p_w$. In particular, they coincide
for the lattice polynomial $p_w\wedge f_k$. Since
$$
(\pi_\omega\wedge f_k)(\vec\kappa)=1 \quad\Leftrightarrow\quad \pi_\omega(\vec\kappa)=1 \quad\mbox{and}\quad |\vec\kappa|\geqslant n-k+1,
$$
it follows that
\begin{eqnarray*}
\lefteqn{\sum_{s=n-k+1}^n\frac 1{{n\choose s}}\sum_{\textstyle{\vec\kappa\in\{0,1\}^n\atop
|\vec\kappa|=s}}\pi_\omega(\vec\kappa)\sum_{\textstyle{\vec\kappa\in\{0,1\}^n\atop |\vec\kappa|=s}}\Pr(\vec\chi(y)=\vec\kappa)}\\
& = & \sum_{s=n-k+1}^n\sum_{\textstyle{\vec\kappa\in\{0,1\}^n\atop |\vec\kappa|=s}} \pi_\omega(\vec\kappa)\,\Pr(\vec\chi(y)=\vec\kappa),\qquad
k=0,\ldots,n.
\end{eqnarray*}
It implies that for any lattice polynomial $p_w$ it must hold true that
\begin{eqnarray}
\lefteqn{\frac 1{{n\choose s}}\sum_{\textstyle{\vec\kappa\in\{0,1\}^n\atop
|\vec\kappa|=s}}\pi_\omega(\vec\kappa)\sum_{\textstyle{\vec\kappa\in\{0,1\}^n\atop |\vec\kappa|=s}}\Pr(\vec\chi(y)=\vec\kappa)}\nonumber\\
& = & \sum_{\textstyle{\vec\kappa\in\{0,1\}^n\atop |\vec\kappa|=s}} \pi_\omega(\vec\kappa)\,\Pr(\vec\chi(y)=\vec\kappa),\qquad
s=0,\ldots,n.\label{eq:33DerEq}
\end{eqnarray}
Fix $s\in\{0,\ldots,n\}$ and $K\subseteq [n]$ such that $|K|=s$. Applying (\ref{eq:33DerEq}) to the lattice polynomial
$\pi_\omega(\vec\kappa)=\wedge_{i\in K}\kappa_i$, it yields that
$$
\frac 1{{n\choose s}}\sum_{\textstyle{\vec\kappa\in\{0,1\}^n\atop |\vec\kappa|=s}}\Pr(\vec\chi(y)=\vec\kappa) = \Pr(\vec\chi(y)=\vec\kappa')
$$
for any $\vec\kappa'\in\{0,1\}^n$ such that $|\vec\kappa'|=s$, which completes the proof.
\end{proof}



\end{document}